\documentclass[11pt]{amsart}
\usepackage{xcolor}
\usepackage{amsmath}
\usepackage{amsthm}
\usepackage{lipsum}
\usepackage{amsopn}
\usepackage{amssymb}
\usepackage{fullpage}
\usepackage{enumerate}
\numberwithin{equation}{section}
\usepackage{float}
\usepackage{graphicx}
\usepackage{titlesec}
\usepackage{dsfont}
\usepackage{fancyhdr}
\usepackage{hyperref}
\usepackage{textcomp}
\usepackage{fancyhdr}
\newtheorem{corollary}{Corollary}[section]
\newtheorem{lemma}{Lemma}[section]
\newtheorem{theorem}{Theorem}[section]
\newtheorem{proposition}{Proposition}[section]

\theoremstyle{definition}
\newtheorem{remark}{Remark}[section]
\newtheorem{example}{Example}[section]

\newtheorem{thm}{Theorem}

\DeclareMathOperator{\Arg}{Arg}

\DeclareMathOperator{\D}{\mathbb{D}}
\DeclareMathOperator{\RE}{Re}
\DeclareMathOperator{\IM}{Im}
\titleformat{\section}
  {\normalfont\scshape}{\thesection}{1em}{\centering}
  \titleformat{\subsection}[runin]
  {\bfseries}{\thesubsection}{1em}{}

\begin{document}

\title{Rates of convergence for holomorphic semigroups of finite shift}

\author{Maria Kourou$^{\S}$}
\thanks{$^{\S}$Partially supported by the Alexander von Humboldt Foundation.}  
\address{Department of Mathematics, Julius-Maximilians University of W\"urzburg, 97074, W\"urzburg, Germany}
\email{maria.kourou@uni-wuerzburg.de}
\author{Eleftherios K. Theodosiadis}
\address{Department of Mathematics, Stockholm University, 106 91, Stockholm, Sweden}
\email{eleftherios@math.su.se}

\author{Konstantinos Zarvalis}
\address{Department of Mathematics, Aristotle University of Thessaloniki, 54124, Thessaloniki, Greece}
\email{zarkonath@math.auth.gr}

\fancyhf{}
\renewcommand{\headrulewidth}{0pt}
\fancyhead[RO,LE]{\small \thepage}
\fancyhead[CE]{\footnotesize }
\fancyhead[CO]{\footnotesize } 

\fancyfoot[L,R,C]{}
\date{}

\begin{abstract}
    We study parabolic semigroups of finite shift in the unit disk with regard to their geometric properties and the rate of convergence of their orbits to the Denjoy--Wolff point. We examine this rate in terms of Euclidean distance, hyperbolic distance and harmonic measure. We further discuss the corresponding rates of convergence for parabolic semigroups of positive hyperbolic step and infinite shift.
\end{abstract}

\maketitle

\section{Introduction}
The study of continuous semigroups of holomorphic functions in the unit disk (or from now on \textit{semigroups in $\D$}) began during the nineteenth century and has since offered a fruitful theory with various applications in several fields of mathematics. The starting point of the modern research on semigroups resulted through the works of Berkson and Porta in \cite{berksonporta}. Formally, a semigroup $(\phi_t)$ in $\D$ is a family of holomorphic functions $\phi_t:\D\to\D$, $t\ge0$, which satisfies the following three conditions:

\begin{enumerate}
    \item[(i)] $\phi_0={\textup{id}}_{\D}$,
    \item[(ii)] $\phi_{t+s}=\phi_t\circ\phi_s$, for all $t,s\ge0$,
    \item[(iii)] $\lim_{t\to s}\phi_t(z)=\phi_s(z)$, for all $z\in\D$ and all $s\ge0$.
\end{enumerate}
For a profound presentation of the rich theory of semigroups, the interested reader may refer to \cite{Booksem, AnalyticFlows, shoikhet}.

Given $z\in\D$, the set $\{\phi_t(z):t\ge0\}$ and the curve $\gamma_z:[0,+\infty)\to\D$ with $\gamma_z(t)=\phi_t(z)$ are interchangeably called the \textit{(forward) orbit} (also seen as \textit{trajectory}) of $z$, while $z$ is its \textit{starting point}. Suppose $(\phi_t)$ is a semigroup in $\D$, where for some $t_0 >0$ the iterate $\phi_{t_0}$ is not a conformal automorphism of $\D$ with a fixed point in $\D$. 
Naturally, we are interested in the asymptotic behavior of the orbits, as $t\to+\infty$. The definitive step towards this direction is given by means of the Denjoy--Wolff Theorem (see \cite[Theorem 8.3.1]{Booksem}), according to which there exists a unique $\tau\in\overline{\D}$ such that
$$\lim\limits_{t\to+\infty}\phi_t(z)=\lim\limits_{t\to+\infty}\gamma_z(t)=\tau, \quad\text{for all }z\in\D.$$
This unique for each semigroup point $\tau$ is called the \textit{Denjoy--Wolff point} of the semigroup and its position in the closure of the unit disk provides a first classification within the class of semigroups:
\begin{enumerate}
    \item[(i)] if $\tau\in\D$, then $(\phi_t)$ is characterized as \textit{elliptic},
    \item[(ii)] whereas if $\tau\in\partial\D$, then $(\phi_t)$ is characterized as \textit{non-elliptic}.
\end{enumerate}
During the course of this article, we exclusively deal with non-elliptic semigroups. The fact that their Denjoy--Wolff points are positioned on the unit circle gives birth to a deep theory with interesting results and intricacies. 

In addition to the limit of the orbits itself, it is usual to inquire about the manner of this convergence: the type of the convergence (tangential or non-tangential) or even the very angle by which the orbits land at the Denjoy--Wolff point. A helpful tool towards this goal are the \textit{horodisks} of $\D$. A horodisk $E(\tau,R)$ of $\D$, where $\tau\in\partial\D$ and $R>0$, is a disk of radius $\frac{R}{R+1}$ contained inside $\D$ which is internally tangent to $\partial\D$ at the point $\tau$. To be more concrete, we may define
$$E(\tau,R):=\left\{z\in\D:\frac{|\tau-z|^2}{1-|z|^2}<R\right\}.$$
Through the study of horodisks, we are led to a further distinction within the class of non-elliptic semigroups. A non-elliptic semigroup $(\phi_t)$ in $\D$ with Denjoy--Wolff point $\tau$ is said to be of \textit{finite shift} if for any $z\in\D$, there exists some horodisk $E(\tau,R_z)$, where $R_z$ depends on the starting point $z$, such that the orbit of $z$ does not intersect this horodisk (i.e. $\phi_t(z)\notin E(\tau,R_z)$, for all $t\ge0$). If this condition does not hold, we say that $(\phi_t)$ is of \textit{infinite shift}. For a detailed presentation of holomorphic semigroups of finite shift we refer to \cite[Section 17.7]{Booksem} and the articles \cite{BetsDescript, cordella, ElKhReSh_2010, karamanlis}.

It is quite straightforward from the definition that all the orbits of a semigroup of finite shift must converge to the Denjoy--Wolff point tangentially. This is because each orbit $\gamma_z$, $z\in\D$, must reach the Denjoy--Wolff point $\tau$ traveling between the circles $\partial\D$ and $\partial E(\tau,R_z)$. Furthermore, since $\gamma_z$ stays outside the horodisk $E(\tau,R_z)$, one may roughly say that $\gamma_z$ remains ``very close'' to the unit circle. Hence, semigroups of finite shift are also called \textit{strongly tangential} in the literature.

This general idea about the asymptotic behavior of the orbits in a semigroup of finite shift motivates the current work towards the study of their rate of convergence to the Denjoy--Wolff point.
More specifically, we examine the rate of convergence of orbits for non-elliptic semigroups of either finite or infinite shift. In order to discuss about infinite shift, we restrict to the class of \textit{parabolic semigroups of positive hyperbolic step}, namely those semigroups where the angular derivative $\phi_t^{\prime}(\tau)=1$ for all $t\geq 0$ and the hyperbolic distance $\lim_{t\to \infty} d_{\D}(\phi_t(z), \phi_{t+s}(z)) >0 $, for all $z \in \D$ and for all $s>0 $; for the proper definition and a geometric interpretation we refer to Subsection \ref{sub:Koenigs}.

For every parabolic semigroup of positive hyperbolic step, as outlined in Proposition \ref{sectorprop}, we find that there exists an intrinsic quantity of the semigroup, which we call \textit{inner argument} of $(\phi_t)$ and denote it by $\Theta \in [0,\pi]$. As it turns out, for semigroups of finite shift it holds $\Theta=\pi$, whereas for those of infinite shift $\Theta$ can take any value in $[0,\pi]$.
More information on the inner argument of $(\phi_t)$ follows in Section \ref{sec:finite}.

To study the rate of convergence, the first step is to investigate the Euclidean distance $|\phi_t(z)-\tau|$ and how fast it converges to zero, as $t\to +\infty$. It should be mentioned that the asymptotic behavior of $|\phi_t(z)-\tau|$, as $t \to +\infty$, has been extensively studied, during the past two decades, for non-elliptic semigroups in accordance to the properties of $(\phi_t)$; see e.g. \cite{betshypsem, Bets-Contreras, JacLevR_2011}, \cite[Chapter 7]{shoikhetelin} and \cite[Chapter 16]{Booksem}.
Namely, Betsakos \cite{betsparsem} proved that given a non-elliptic semigroup $(\phi_t)$ with Denjoy--Wolff point $\tau\in\partial\D$
$$|\phi_t(z)-\tau|\le\frac{C}{\sqrt{t}}\, , \quad  \text{for all }t>0\text{ and all }z\in\D,$$
where the constant $C$ depends on the semigroup and the starting point $z$. In fact, better upper bounds as well as lower bounds can be obtained, when restricting to certain types of semigroups. 
We obtain the following result. 


\begin{theorem}\label{theo-eucl}
    Let $(\phi_t)$ be a parabolic semigroup in $\D$ of positive hyperbolic step with Denjoy-Wolff point $\tau\in\partial\D$ and inner argument $\Theta\in(0,\pi]$. The following are true:
    \begin{itemize}
    \item[(i)]If $(\phi_t)$ is of infinite shift, then for every $z\in\D$ and for every $\epsilon>0$, there exist two positive constants $c_1=c_1(z,\epsilon)$ and $c_2=c_2(z)$ such that
    $$\frac{c_1}{t^{\frac{\pi+\Theta}{\Theta}+\epsilon}}\le|\phi_t(z)-\tau|\le\frac{c_2}{t}, \quad\text{for all }t>1.$$
   
    \item[(ii)]If $(\phi_t)$ is of finite shift (and hence $\Theta=\pi$), then for every $z\in\D$, there exist two positive constants $c_1=c_1(z)$ and $c_2=c_2(z)$ such that
    $$\frac{c_1}{t}\le|\phi_t(z)-\tau|\le\frac{c_2}{t}, \quad\text{for all }t>1.$$
    \end{itemize}
\end{theorem}

The upper bound for $|\phi_t(z)-\tau|$ is already known from results found in \cite{betsparsem} and \cite{Bets-Contreras} in conjunction with the geometric properties of $(\phi_t)$. Further details can be found in Subsection \ref{sub:ratesofconv}. 
The lower bound for $|\phi_t(z)-\tau|$, in the case of a non-elliptic semigroup $(\phi_t)$ of finite shift can be also derived by results in \cite{cordella} in combination with the theory of speeds of convergence for non-elliptic semigroups; see e.g. \cite[Chapter 16]{Booksem}. However, we follow a different and more geometric approach.

We further express the hyperbolic rate for non-elliptic semigroups of prescribed inner argument, due to its close relation to the Euclidean rate. More specifically, we obtain the following result on the rate of convergence with the use of hyperbolic distance $d_{\mathbb{D}}$ in the unit disk $\D$. 

\begin{theorem}\label{theo-hyper}
    Let $(\phi_t)$ be a parabolic semigroup in $\D$ of positive hyperbolic step with Denjoy-Wolff point $\tau\in\partial\D$ and inner argument $\Theta\in(0,\pi]$. The following are true:
    \begin{itemize}
        \item[(i)] If $(\phi_t)$ is of infinite shift, then for every $z\in\D$ and for every $\epsilon>0$, there exist two positive constants $c_1=c_1(z,\epsilon)$ and $c_2=c_2(z)$ such that
     $$\log t-c_2\le d_{\D}(z,\phi_t(z))\le \left(\frac{\pi+\Theta}{2\Theta} + \epsilon \right) \log t +c_1, \quad\text{for all }t>1.$$
     \item[(ii)] If $(\phi_t)$ is of finite shift (and hence $\Theta=\pi$), then for every $z\in\D$, there exist two positive constants $c_1=c_1(z)$ and $c_2=c_2(z)$ such that
    $$\log t-c_2\le d_{\D}(z,\phi_t(z))\le \log t +c_1, \quad\text{for all }t>1.$$  
    \end{itemize}    
\end{theorem}

Again, the inequality in the case of finite shift may also be implied through a combination of results found in \cite{cordella} and \cite{Booksem}.

Last but not least, we estimate the rate of convergence for parabolic semigroups of positive hyperbolic step through the scope of potential theory and in particular, in terms of the harmonic measure. The theory of harmonic measure is analyzed in more depth in Subection \ref{sub:harmonicmeas}, but we provide a brief explanation in order to state our last result. 
Consider a simply connected domain $\Omega\subsetneq\mathbb{C}$ and a Borel set $E \subseteq \partial\Omega$. Then, the \textit{harmonic measure} of $E$ with respect to $\Omega$ is the harmonic function in $\Omega$ with boundary values equal to the indicator function $\chi_E$. Given $z\in\Omega$, we use the notation $\omega(z,E,\Omega)$ for the respective harmonic measure. 

Recall that given a semigroup $(\phi_t)$ of finite shift and $z\in\D$, then the orbit $\gamma_z$ converges to the Denjoy--Wolff point $\tau$ strongly tangentially without intersecting all the sufficiently small horodisks tangent at $\tau$. As we already mentioned, we understand that in order for this behavior to take place, the orbit must remain adequately close to $\partial\D$. Therefore, by denoting by $E_1$ and $E_2$ the open half-circles defined by $\pm\tau$, one can naturally obtain that $\omega(\phi_t(z),E_1,\D)\to0$ and $\omega(\phi_t(z),E_2,\D)\to1$, as $t\to+\infty$, or the reverse (depending on if the orbit converges to $\tau$ with angle $0$ or with angle $\pi$). 
Our main intention is to find out how fast this convergence is.

\begin{theorem}\label{theo-harm}
    Let $(\phi_t)$ be a parabolic semigroup in $\D$ of positive hyperbolic step with Denjoy--Wolff point $\tau \in \partial\D$ and inner argument $\Theta\in(0,\pi]$. Suppose that the points $\pm\tau$ separate $\partial\D$ into the two open subarcs $E_1$ and $E_2$. Fix $z\in\D$ and assume that $E_1$ is the subarc that is closer to $\gamma_z$, for large $t\ge0$. Then there exists a positive constant $c=c(z)$ such that
    $$\limsup\limits_{t\to+\infty}\left(t\omega(\phi_t(z),E_2,\D)\right)\le\frac{c}{\Theta}.$$
\end{theorem}
The correct choice between $E_1$ and $E_2$ and which one is ``closer'' to $\gamma_z$ are explained thoroughly in the subsequent sections. Moreover, we will see that the bound is sharp. Furthermore, we will examine the behavior of the harmonic measure in a case where $\Theta=0$.

One may observe that Theorems \ref{theo-hyper} and \ref{theo-harm} are stated in terms of conformally invariant quantities. This conformal invariance plays a pivotal role in their proofs. Indeed, we are going to prove both theorems by means of a suitable function $h$ (more on that in Section \ref{sec:tools}) and passing to a simply connected domain, where our calculations are facilitated by the simple geometry of $h\circ\gamma_z$. On the other hand, Theorem \ref{theo-eucl} lacks this invariance, thus necessitating a restriction to the unit disk. 

The structure of the present article is as follows: In Section \ref{sec:tools}, we provide all the necessary background on semigroups and conformal invariants that deems useful for the proofs of our theorems. Next, in Section \ref{sec:finite}, we introduce the inner argument and then state and prove geometric results concerning semigroups with regard to their inner argument. These results are used later on the rates of convergence. Afterwards, Section \ref{sec:rates} is devoted to the proofs of Theorems \ref{theo-eucl} and \ref{theo-hyper}. Finally, in Section \ref{sec:harmmeas}, we deal exclusively with harmonic measure by proving Theorem \ref{theo-harm} and discussing some extremal cases.

\section{Necessary Tools and Information}\label{sec:tools}

\subsection{ Koenigs Function}\label{sub:Koenigs}
To start with, except for the Denjoy--Wolff point, there is a second powerful tool which facilitates the study of semigroups and their orbits. Indeed, let $(\phi_t)$ be a non-elliptic semigroup in $\D$. Then there exists a unique conformal mapping $h:\D\to\mathbb{C}$ with $h(0)=0$ such that
$$h(\phi_t(z))=h(z)+t, \quad\text{for all }z\in\D\text{ and all }t\ge0 \, ;$$
for further details see e.g  \cite[Chapter 9]{Booksem} and references therein.  
This unique $h$ is called the \textit{Koenigs function} of the semigroup, while the simply connected domain $\Omega:=h(\D)$ its \textit{Koenigs domain} (also seen as \textit{associated planar domain} in the literature).

It is clear from the definition that each orbit $\gamma_z$ is mapped through $h$ onto the horizontal half-line $\{h(z)+t:t\ge0\}$ that stretches towards infinity in the positive direction (i.e. with constant imaginary part and increasing real part). As a result, the Koenigs function allows us to linearize the orbits of a semigroup, thus rendering their study simpler. Furthermore, the definition of $h$ implies at once that $\Omega$ is a convex in the positive direction domain. This signifies that $\Omega+t\subseteq\Omega$, for all $t\ge0$.

We understand that each non-elliptic semigroup corresponds through its Koenigs function to a unique convex in the positive direction simply connected domain. An important piece of information is that the converse is true, as well. More specifically, given any such domain $\Omega$ and considering a Riemann mapping $h:\D\to\Omega$ with $h(0)=0$, then there exists a unique non-elliptic semigroup $(\phi_t)$ defined through the relation $\phi_t(z)=h^{-1}(h(z)+t)$, $z\in\D$, $t\ge0$. Consequently, non-elliptic semigroups come into a one-to-one correspondence with convex in the positive direction simply connected domains. The importance of this observation lies on the fact that it frequently allows us to study the geometric properties of the Koenigs domain $\Omega$ in order to extract conclusions about the behavior of the corresponding semigroup. Moreover, it allows us to construct desired counterexamples. In fact, when we want to disprove a conjecture through a counterexample, instead of constructing a semigroup $(\phi_t)$, we just choose a suitable domain $\Omega$ and work with its geometry. Then, the non-elliptic semigroup resulting from $\Omega$ is the desired one.

Moreover, the Koenigs function $h$ leads to a further classification of non-elliptic semigroups. Since the Koenigs domain $\Omega$ is convex in the positive direction, it is natural to check whether it might be contained in certain standard horizontal domains:
\begin{enumerate}
    \item[(i)] if $\Omega$ is contained inside a horizontal strip, then $(\phi_t)$ is characterized as \textit{hyperbolic},
    \item[(ii)] if $\Omega$ is contained inside a horizontal half-plane, but not in a horizontal strip, then $(\phi_t)$ is characterized as \textit{parabolic of positive hyperbolic step},
    \item[(iii)] otherwise $(\phi_t)$ is characterized as \textit{parabolic of zero hyperbolic step}.
\end{enumerate}

The actual definitions of hyperbolic semigroups, parabolic semigroups and hyperbolic step are different and rely on angular derivatives and hyperbolic distance. 
Further information on the characterization of semigroups, Koenigs functions and domains can be found in \cite{Booksem}.

\subsection{Infinitesimal Generator}
Let $(\phi_t)$ be a semigroup in $\D$. Then there exists a unique holomorphic map $G:\D\to\mathbb{C}$ such that
$$\frac{\partial\phi_t(z)}{\partial t}=G(\phi_t(z)), \quad z\in\D, \quad t\ge0.$$
This mapping is called the \textit{infinitesimal generator} of the semigroup and is inextricably linked to the study of semigroups. Indeed, its properties reveal a lot of information about the associated semigroup. A useful relation that intertwines the infinitesimal generator $G$ and the Koenigs function $h$ of a non-elliptic semigroup is the following:
$$G(z)=\frac{1}{h'(z)}, \quad\text{for all }z\in\D.$$
In this present article, we will not really delve into the rich theory of infinitesimal generators. We will rather only need the above two relations.

\subsection{ Rates of Convergence}\label{sub:ratesofconv}
A fundamental field of research on non-elliptic semigroups of holomorphic functions in $\D$ is the rate of convergence of the orbits. It is of great interest to study how fast an orbit $\gamma_z$ is pushed away from the starting point $z$ or how fast it is pulled towards the Denjoy--Wolff point $\tau\in\partial\D$. As we mentioned in the Introduction, a remarkable result concerning the latter rate that first appeared in \cite{betsparsem} is that any non-elliptic semigroup $(\phi_t)$ in $\D$ satisfies 
$$|\phi_t(z)-\tau|=|\gamma_z(t)-\tau|\le\frac{C}{\sqrt{t}},\quad\text{for all }z\in\D\text{ and all }t>0,$$
where the constant $C>0$ depends on the starting point $z$. In the same article, the author improves this bound when restricting to more exclusive types of semigroups. Indeed, the rate $t^{-\frac{1}{2}}$ may be changed to $t^{-1}$ for parabolic semigroups of positive hyperbolic step and to $e^{-\mu t}$ for hyperbolic semigroups, where $\mu$ is a positive number depending on the semigroup and is called its \textit{spectral value} (for more information we refer the interested reader to \cite[Definition 8.3.2]{Booksem}). However, the bound $t^{-\frac{1}{2}}$ is the best we can expect in the general case, in the sense that there exists a parabolic semigroup of zero hyperbolic step whose orbits acquire this rate. The existence of a lower bound for the rate of convergence has also been explored in \cite{Bets-Contreras} in the cases where the Koenigs domain satisfies certain geometric properties.

\subsection{ Hyperbolic Geometry} In the current subsection, we proceed to a deeper examination of certain hyperbolic quantities such as the hyperbolic metric, density and distance in the unit disk. For a profound presentation of the rich theory of hyperbolic geometry we refer to \cite{bearmin}, \cite[Chapter 5]{Booksem} and references therein.

The \textit{hyperbolic metric} in the unit disk $\D$ is given by the formula
$$\lambda_{\D}(z)|dz|=\frac{|dz|}{1-|z|^2}, \quad z\in\D,$$
where the function $\lambda_{\D}$ is called the \textit{hyperbolic density} of $\D$. Let $\gamma:(\alpha,\beta)\to\D$ be a piecewise smooth curve. Then, for $\alpha<t_1\le t_2<\beta$, the hyperbolic length $l_{\D}(\gamma;[t_1,t_2])$ of $\gamma$ between its points $\gamma(t_1)$ and $\gamma(t_2)$ is given through the formula 
$$l_{\D}(\gamma;[t_1,t_2])=\int\limits_{t_1}^{t_2}\lambda_{\D}(\gamma(t))|\gamma'(t)|dt.$$
For $z,w\in\D$, their \textit{hyperbolic distance} in $\D$ induced by the corresponding hyperbolic metric is given through
$$d_{\D}(z,w)=\inf\limits_{\gamma}\int\limits_{\gamma}\lambda_{\D}(\zeta)|d\zeta|,$$
where the infimum is taken over all piecewise smooth curves $\gamma$ joining $z$ to $w$ inside the unit disk. Evidently, given a piecewise smooth curve $\gamma:[0,1]\to\D$ with $\gamma(0)=z$ and $\gamma(1)=w$, it is true that $d_{\D}(z,w)\le l_{\D}(\gamma;[0,1])$. Equality holds when $\gamma$ is a \textit{hyperbolic geodesic}. In fact, it can be computed that the hyperbolic distance may be given through the formula
$$d_{\D}(z,w)=\frac{1}{2}\log\frac{|1-\overline{z}w|+| z-w|}{|1-\overline{z}w|-| z-w|}.$$

Let $\Omega\subsetneq\mathbb{C}$ be a simply connected domain and consider $f:\Omega\to\D$ to be a Riemann mapping. Then, the \textit{hyperbolic metric} in $\Omega$ is given by
$$\lambda_\Omega(z)|dz|=\lambda_{\D}(f(z))|f'(z)||dz|, \quad z\in\Omega.$$
Again, the function $\lambda_\Omega$ is called the \textit{hyperbolic density} of $\Omega$. It can be proved that the definition is independent of the choice of the Riemann mapping $f$.

For $z,w\in\Omega$, their \textit{hyperbolic distance} in $\Omega$ may be given via the relation
$$d_\Omega(z,w)=d_{\D}(f(z),f(w)).$$
Hence, the hyperbolic distance is conformally invariant, allowing us to translate problems from one domain to another. 
An important property of hyperbolic distance is its \textit{domain monotonicity property}. Given two simply connected domains $\Omega_1$ and $\Omega_2$ with $\Omega_1\subseteq \Omega_2 \subsetneq \mathbb{C}$, it follows
$$d_{\Omega_1}(z,w)\ge d_{\Omega_2}(z,w), \quad\text{for all }z,w\in\Omega_1.$$


Before we end the subsection, we give explicit formulas for the hyperbolic distance in two specific domains, needed for our proofs. These formulas may be deduced by the conformal invariance of the hyperbolic distance and the simplicity of the respective Riemann mappings.

Firstly, let $\rho\in\mathbb{R}$ and consider the upper horizontal half-plane $H_{\rho}:=\{\zeta\in\mathbb{C}:\IM \zeta>\rho\}$. Then,
\begin{equation}\label{half-plane}
 d_{H_{\rho}}(z,w)=\frac{1}{2}\log\frac{\left|z-\overline{w}-2\rho i\right|+\left|z-w\right|}{\left|z-\overline{w}-2\rho i\right|-\left|z-w\right|}, \quad\text{for all }z,w\in H_{\rho}.   
\end{equation}

In a similar fashion, we may write down the hyperbolic distance in a lower horizontal half-plane $H_{\rho}^-:=\{\zeta\in\mathbb{C}:\IM \zeta<\rho\}$. Lastly, let $\theta\in(0,\pi]$ and consider the \textit{upper horizontal angular sector} $S_{\theta}:=\{\zeta\in\mathbb{C}:\Arg\zeta\in(0,\theta)\}$. Then,
\begin{equation}\label{sector}
d_{S_{\theta}}(z,w)=\frac{1}{2}\log\frac{\left|z^{\frac{\pi}{\theta}}-\overline{w}^{\frac{\pi}{\theta}}\right|+\left|z^{\frac{\pi}{\theta}}-w^{\frac{\pi}{\theta}}\right|}{\left|z^{\frac{\pi}{\theta}}-\overline{w}^{\frac{\pi}{\theta}}\right|-\left|z^{\frac{\pi}{\theta}}-w^{\frac{\pi}{\theta}}\right|}, \quad\text{for all }z,w\in S_{\theta}.
\end{equation}

A similar computation could be made for a lower horizontal angular sector $S_{\theta}^-:=\{\zeta\in\mathbb{C}:\Arg\zeta\in(-\theta,0)\}$. We use the notations $H_{\rho}, H_{\rho}^-, S_{\theta}$ and $S_{\theta}^-$ frequently during the course of the following sections. For the sake of convenience, we use the notation $\mathbb{H}$ for the standard upper half-plane $H_0$.

\subsection{ Harmonic Measure} \label{sub:harmonicmeas}
A second conformal invariant quantity we use throughout the course of the current article is the \textit{harmonic measure}. Even though it might not seem that way at first glance, harmonic measure can be utilized as a tool in a multitude of different areas of mathematics producing stunning results. For an in depth exploration of its capabilities see \cite[Chapter 7]{Booksem}, \cite{margarnett} or \cite{ransford}.

Having already defined harmonic measure in the Introduction, we pass to several properties we use in the proof of Theorem \ref{theo-harm}. First of all, the harmonic measure is conformally invariant. Therefore, if $\Omega_1,\Omega_2$ are two simply connected domains in $\mathbb{C}$ and $f:\Omega_1\to\Omega_2$ is conformal, then
$$\omega(z,E,\Omega_1)=\omega(f(z),f(E),\Omega_2), \quad\text{for all }z\in\Omega_1,$$
for each Borel subset $E$ of $\partial\Omega_1$. Note here that in the case of non-Jordan domains, the sets $E,f(E)$ are defined in the Carath\'{e}odory sense and consist of prime ends.

Moreover, for each Borel subset $E$ of $\Omega$ the function $\omega(\cdot,E,\Omega)$ is a harmonic function in $\Omega$,  while simultaneously for each $z\in\Omega$ the function $\omega(z,\cdot,\Omega)$ is a Borel probability measure on $\partial\Omega$. 

In addition, the harmonic measure conceals a domain monotonicity property as well. To be specific, let $\Omega_1\subset\Omega_2\subsetneq\mathbb{C}$ be simply connected and $E\subset\partial\Omega_1\cap\partial\Omega_2$ be a Borel set. Then
$$\omega(z,E,\Omega_1)\le\omega(z,E,\Omega_2)\, , \quad \text{for all }z\in\Omega_1.$$
This monotonicity property can be stated more precisely by means of the so-called \textit{Strong Markov Property} (see \cite[p.88]{portstone}). In particular, given $\Omega_1,\Omega_2$ and $E$ as above, we have for $z\in\Omega_1$
$$\omega(z,E,\Omega_2)=\omega(z,E,\Omega_1)+\int\limits_{\partial\Omega_1\cap \Omega_2}\omega(\zeta,E,\Omega_2)\cdot\omega(z,d\zeta,\Omega_1).$$


In Section \ref{sec:harmmeas} we use the formula for the harmonic measure inside an angular sector. Let $U=\{\zeta\in\mathbb{C}:\Arg\zeta\in(\alpha,\beta)\}$, for $-\pi\le\alpha<\beta\le\pi$. According to \cite[p.100]{ransford}, it holds
$$\omega(z,\{\Arg\zeta=\beta\},U)=\frac{\Arg z-\alpha}{\beta-\alpha}=1-\omega(z,\{\Arg\zeta=\alpha\},U), \quad z\in U.$$

\subsection{ Conformality at the Boundary}\label{sub:conformality}

Our first task will be to give a geometric description of the Koenigs domain of a semigroup in $\D$ of finite shift. This description will involve the classical notions of angular limits and conformality at the boundary. We mostly follow \cite{Jenkins} as far as the notation is concerned. 

Suppose that $f$ maps an upper half-plane conformally into some other upper half-plane. The theory can be written for any kind of half-plane, but our study concerns solely horizontal half-planes and especially the upper ones. Suppose that $\angle\lim_{z\to\infty}f(z)=\infty$ and $\angle\lim_{z\to\infty}\frac{z}{f(z)}=\sigma$. Then  $\sigma$ is called the \textit{angular derivative} of $f$ at $\infty$. If, in addition, $\sigma\ne0,\infty$, then we say that $f$ is \textit{conformal} at $\infty$.

Now consider a conformal mapping $h:\D\to\mathbb{C}$ such that the image $\Omega=h(\D)$ is contained inside an upper half-plane. Suppose that $\angle\lim_{z\to\tau}h(z)=\infty$ for some $\tau\in\partial\D$. Take also $C:\D\to\mathbb{H}$ to be the Cayley transform with $C(z)=i\frac{\tau+z}{\tau-z}$. Evidently $C(\tau)=\infty$. We say that $h$ is conformal at $\tau$ if and only if $h\circ C^{-1}$ is conformal at $\infty$.

Note here that the conformality of $f$ at $\infty$ or of $h$ at $\tau$ is actually a geometric property of $\Omega=h(\D)$. Indeed, there exists a geometric characterization of conformality involving extremal length. However, for the purposes of the present work we will only need a looser condition, which is necessary, but not sufficient. Suppose that $f:\mathbb{H}\to\mathbb{C}$ with $f(\mathbb{H})\subset H_\rho$, $\rho\in\mathbb{R}$, is conformal at $\infty$. Let $\epsilon\in(0,\frac{\pi}{2})$. Then, there exists some $R>0$ depending on $\epsilon$ so that the angular domain $\{\zeta\in\mathbb{C}:|\zeta|>R, \;\Arg \zeta\in(\epsilon,\pi-\epsilon)\}+i\rho$ is contained in $f(\mathbb{H})$ (cf. \cite[p.94]{Jenkins}).

\section{Properties of Finite Shift}\label{sec:finite}

The main objective of this section is to comprehend the geometry of the Koenigs domain of a semigroup of finite shift. Recall that semigroups of finite shift are a priori non-elliptic, since the definition requires convergence to a boundary point. Moreover, given a semigroup $(\phi_t)$ of finite shift with Denjoy--Wolff point $\tau$, for every $z\in\D$ there exists a horodisk $E(\tau,R_z)$ such that $\gamma_z$ does not intersect this horodisk (or any smaller horodisk for that matter). On the other hand, every Stolz angle with vertex $\tau$ is eventually contained in any horodisk $E(\tau,R)$, $R>0$. As a result, all the orbits of a semigroup of finite shift converge to $\tau$ tangentially. However, by \cite[Lemma 17.4.3]{Booksem} we know that all the orbits of a hyperbolic semigroup converge non-tangentially to $\tau$. So, at once, hyperbolic semigroups are of infinite shift. On top of that, it is proved (see \cite[Proposition 17.7.3]{Booksem}) that parabolic semigroups of zero hyperbolic step are of infinite shift, as well. Therefore, every semigroup of finite shift is necessarily parabolic of positive hyperbolic step.

Before proceeding to results concerning only semigroups of finite shift, we will first provide some lemmas and propositions about the greater class of parabolic semigroups of positive hyperbolic step. In order not to be pedantic, we will always suppose that their Koenigs domain is contained inside some upper half-plane $H_\rho$. Of course, identical results hold in case the Koenigs domain is contained inside a lower half-plane.
\begin{lemma}\label{sectorlemma}
    Let $\Omega$ be a convex in the positive direction simply connected domain contained in some upper horizontal half-plane. Suppose that there exist $p\in\Omega$ and $\theta\in(0,\pi)$ such that $p+S_{\theta}\subset\Omega$. Then, for each $w\in\Omega$, there exists $q_w\in\Omega$ such that $q_w+S_\theta\subset\Omega$ and $w+t$ is eventually contained in $q_w+S_\theta$, as $t\to+\infty$.
\begin{proof}
    First of all, if the whole half-line $p+s$, $s\ge0$, is part of the boundary $\partial\Omega$, then the desired result it trivial with $q_w=p$, for all $w\in\Omega$. So, we may assume that this is not the case.
    
    We need to further distinguish cases depending on the relative position of $w$ and $p$ in order to obtain the result. First of all, if $\IM w>\IM p$, then the result is trivial with $q_w=p$.

    Secondly, assume that $\IM w=\IM p$. By the convexity in the positive direction of $\Omega$, the angular sector $p+S_\theta+x$ is contained inside $\Omega$, for all $x>0$. Pick such a $x>0$ to ensure that $\partial(p+S_\theta+x)\cap\partial\Omega=\emptyset$. Therefore, there exists some $\epsilon>0$, sufficiently small, such that $p+S_\theta+x-i\epsilon\subset\Omega$ (this can be deduced in combination with our initial remark in the proof). But this latter sector necessarily eventually contains $w+t$, $t\ge0$. So, the result holds for $q_w=p+x-i\epsilon$. 

    Finally, suppose that $\IM w<\IM p$. Note that since $\Omega$ is convex in the positive direction, we have that
    $$\sup\left\{\RE \zeta:\zeta\in\partial\Omega \text{ and }\IM \zeta\in\left(\IM w,\IM p\right)\right\}<+\infty.$$
    As a consequence, a similar argument as in the previous case, with moving the starting sector sufficiently to the right and then sufficiently downwards, leads to the desired result.
\end{proof}
\end{lemma}

Again, an identical lemma holds when considering sectors of the form $S_{\theta}^-$. However, we will only deal with upper horizontal angular sectors. Observe that the statement of the lemma dictates that $\theta\in(0,\pi)$. Indeed, when $\theta=\pi$, the result fails to hold. Take for example the upper half-plane $\mathbb{H}$ minus some half-line and some $w$ between the real line and the half-line. 


For all parabolic semigroups of positive hyperbolic step, the preceding lemma yields that either $\Omega$ does not contain any upper horizontal angular sector, or for each $z\in\D$, we may find an angular sector of angle $\theta\in (0,\pi)$ depending on $z$ such that it is contained inside $\Omega$, while also eventually containing the half-line $h(z)+t$, as $t\to+\infty$. The question that arises is whether for each $z$ we may find the largest, in terms of angle, such sector. We define the following:
$$\Theta_z:=\sup\{\theta\in(0,\pi]:\text{there exists }p\in\Omega \text{ such that }h(z)+t \text{ is eventually contained in }p+S_\theta\}.$$
If there exists no such sector, we write $\Theta_z:=0$.

\begin{proposition}\label{sectorprop}
    Let $(\phi_t)$ be a parabolic semigroup of positive hyperbolic step in $\D$. Then $\Theta_{z_1}=\Theta_{z_2}$, for all $z_1,z_2\in\D$.
\begin{proof}
    Fix two random distinct $z_1,z_2\in\D$. If $\Theta_{z_1}=0$, then Lemma \ref{sectorlemma} implies $\Theta_{z_2}=0$, as well. Next, we consider the case $\Theta_{z_1}>0$. Aiming for a contradiction, assume that $\Theta_{z_1}>\Theta_{z_2}$. Then, there exists some $\theta\in(\Theta_{z_2},\Theta_{z_1})$ and a point $p\in\Omega$ such that $p+S_\theta\subset\Omega$. Then, by Lemma \ref{sectorlemma} we may find $q\in\Omega$ such that $q+S_\theta\subset\Omega$ and the half-line $h(z_2)+t$ is eventually contained in $q+S_\theta$. Hence, $\theta\le\Theta_{z_2}$. Contradiction! In a reciprocal manner, we may prove that $\Theta_{z_1}<\Theta_{z_2}$ cannot hold either. Consequently, $\Theta_{z_1}=\Theta_{z_2}$. The arbitrariness in the choice of $z_1,z_2\in\D$ implies the desired result.
\end{proof}
\end{proposition}

Therefore, we understand that the supremum of the amplitudes of the horizontal angular sectors contained in $\Omega$ is actually independent of the starting point $z$ and depends on the semigroup itself. For the sake of brevity, we simply call this supremum the \textit{inner argument} of the semigroup and denote it by $\Theta$. 

\begin{remark}
    In this article, the notion of the inner argument of a semigroup $(\phi_t)$ is viewed solely under the scope of parabolic semigroups of positive hyperbolic step. Of course, it makes sense even if we assume that $\Omega$ is contained inside a lower horizontal half-plane. We just have to make a slight modification and take the infimum of negative arguments. Nevertheless, inner argument can be extended to any non-elliptic semigroup $(\phi_t)$. Trivially, every hyperbolic semigroup has inner argument $0$ since the Koenigs domain $\Omega$ fits inside a horizontal strip. On the contrary, for a parabolic semigroup of zero hyperbolic step, the definition of the inner argument may be modified to
    $\Theta=\sup\{\alpha+\beta:\text{for each }z\in\D\text{ there exists }p_z\in\Omega\text{ such that }h(z)+t\text{ is eventually contained in }p_z+S_{\alpha,\beta}\}$, where $S_{\alpha,\beta}=\{w\in\mathbb{C}:-\alpha<\arg w<\beta\}$, for some $0\le\alpha,\beta\le\pi$, $\alpha+\beta>0$. Notice that this more generalized definition of inner argument can be expressed for any non-elliptic semigroup and is equivalent to the one we use in the present work.
\end{remark}

Now, it seems natural to separate parabolic semigroups of positive hyperbolic step into those of finite shift and those of infinite shift, and see their difference with respect to the inner argument. We will first deal with the inner argument of semigroups of finite shift.

The importance of the finite shift is outlined through its connection with the classical angular derivative problem for the associated Koenigs function. Through this connection, we are able to measure the inner argument.

\begin{thm}\cite[Theorem 3]{BetsDescript}\label{thm:angular}
    Let $(\phi_t)$ be a parabolic semigroup with Denjoy--Wolff point $\tau$ and Koenigs function $h$. Then $(\phi_t)$ is of finite shift if and only if $h$ is conformal at $\tau$. 
\end{thm}

\begin{proposition}\label{prop:InnerArgFinite}
    Let $(\phi_t)$ be a semigroup of finite shift. Then its inner argument $\Theta=\pi$. 
    \begin{proof}
    Let $h$ be the Koenigs function and $\Omega$ the Koenigs domain of the semigroup. Since $(\phi_t)$ is of finite shift, it is necessarily parabolic of positive hyperbolic step and thus, $\Omega\subset H_\rho$, for some $\rho<0$. In addition, by Theorem \ref{thm:angular}, $h$ is conformal at the Denjoy-Wolff point $\tau$ of the semigroup. By precomposing with the inverse of the Cayley transform $C:\D\to\mathbb{H}$ with $C(z)=i\frac{\tau+z}{\tau-z}$, we see that $h\circ C^{-1}$ is conformal at $\infty$. By Subsection \ref{sub:conformality}, this implies that for every $\epsilon\in(0,\frac{\pi}{2})$ there exists some $R>0$ such that the set $\{w\in\mathbb{C}:|w|>R, \;\Arg w\in(\epsilon,\pi-\epsilon)\}+i\rho$ is contained inside $\Omega$. But $\Omega$ is convex in the positive direction. Therefore, the larger set $\{w\in\mathbb{C}:|w|>R,\; \Arg w\in(0,\pi-\epsilon)\}+i\rho$ must be contained inside $\Omega$. So, let $\theta\in(\frac{\pi}{2},\pi)$ and set $\epsilon=\pi-\theta$. Then, we may find some $R>0$ such that 
    $$\{w\in\mathbb{C}:|w|>R,\; \Arg w\in(0,\theta)\}+i\rho\subset\Omega.$$
    Consequently, taking $p=i(R+\rho)$, we see that the horizontal angular sector $p+S_\theta$ is contained inside $\Omega$. But this procedure can be replicated for any angle $\theta\in(\frac{\pi}{2},\pi)$. Therefore, $\Theta\ge\theta$ for all $\theta\in(\frac{\pi}{2},\pi)$, which in turn leads to $\Theta=\pi$.
    \end{proof}
\end{proposition}

Nevertheless, this proposition does not provide a characterization for parabolic semigroups of positive hyperbolic step and finite shift. As we will see later on, those of infinite shift can also have inner argument equal to $\pi$. Moreover, the following theorem provides some useful implication concerning angular sectors.

\begin{thm}{\cite[Theorem 17.7.6]{Booksem}}\label{shiftkoenigs}
    Let $(\phi_t)$ be a parabolic semigroup in $\D$ of positive hyperbolic step, with Denjoy--Wolff point $\tau\in\partial\D$ and Koenigs domain $\Omega$.
    \begin{enumerate}
        \item[\textup{(i)}] If $\Omega$ does not contain any horizontal angular sector $p+S_\theta$ or $p+S_{\theta}^-$, $p\in\Omega$, $\theta\in(0,\pi]$, then $(\phi_t)$ is of infinite shift.
        \item[\textup{(ii)}] If $\Omega$ contains a horizontal half-plane, then $(\phi_t)$ is of finite shift.
    \end{enumerate}
\end{thm}

We see that Proposition \ref{prop:InnerArgFinite} essentially strengthens implication (i) of the last theorem. In addition, implication (ii) provides an actual demonstration of Proposition \ref{prop:InnerArgFinite}, since if $\Omega$ contains a half-plane, then clearly $(\phi_t)$ has inner argument $\pi$. However, it would be interesting to know if this implication is actually an equivalence. In other words, is there actually a semigroup $(\phi_t)$ of finite shift whose Koenigs domain does not contain a half-plane? The answer to this question is affirmative and we will see this through the next example which is inspired through \cite{lefteris}, where a similar, but essentially different, construction takes place.

\begin{example}
    For our counterexample, we want a semigroup $(\phi_t)$ of finite shift whose Koenigs domain $\Omega$ does not contain a horizontal half-plane. For the sake of simplicity, we use the upper half-plane $\mathbb{H}$ and its horodisks as our main setting instead of the unit disk. Consider the mapping $h:\mathbb{H}\to\mathbb{C}$ with
    $$h(z)=\sum\limits_{n=1}^{+\infty}\frac{1}{n}\log\left(\frac{n-z}{n-i}\right)-z.$$
    Clearly, we first need to check whether this mapping is a well-defined function of $\mathbb{H}$. Let $K$ be a compact subset of $\mathbb{C}\setminus\{n : \, n\ge 1\}$. Due to compactness there exists some $z_0\in K$ and some $n_0=n_0(K)\in\mathbb{N}$ such that 
    $$\left|\frac{z-i}{n-i}\right|\le\left|\frac{z_0-i}{n-i}\right|=:\rho(n)<1, \quad\text{for all }z\in K\text{ and all }n\ge n_0.$$
    Then, through the Taylor expansions of $\log\left(\frac{n-z}{n-i}\right)$ about $i$ and of $\log(1-x)$ about $0$, we get
    $$\left|\log\left(\frac{n-z}{n-i}\right)\right|=\left|\sum\limits_{k=1}^{+\infty}\frac{1}{k}\left(\frac{z-i}{n-i}\right)^k\right|\le\sum\limits_{k=1}^{+\infty}\frac{1}{k}\left|\frac{z_0-i}{n-i}\right|^k=-\log(1-\rho(n)),$$
    for all $z\in K$ and all $n\ge n_0$. However one can directly see that $-n\log(1-\rho(n))\to|z_0-i|$. As a consequence, there exists a constant $M>0$ depending only on $K$ such that $|-n\log(1-\rho(n))|<M$, for all $n\ge n_0$. Hence
    $$\left|\sum\limits_{n=n_0}^{+\infty}\frac{1}{n}\log\left(\frac{n-z}{n-1}\right)\right|\le\sum\limits_{n=n_0}^{+\infty}\left|\frac{-n\log(1-\rho(n))}{n^2}\right|\le M\sum\limits_{n=n_0}^{+\infty}\frac{1}{n^2} < +\infty .$$
 As a result, the series is locally convergent in $\mathbb{C}\setminus\{n : n\ge1\}$ and thus $h$ is well-defined and analytic in $\mathbb{H}$, while also being continuous on $\mathbb{R}\setminus\{n : n\ge1\}$.

   \begin{figure}
        \centering
        \includegraphics[scale=0.56]{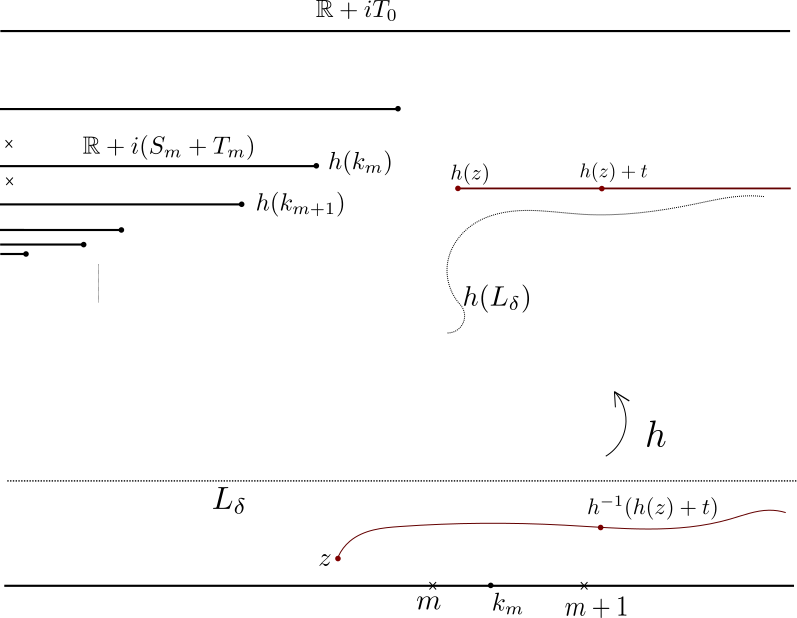}
        \caption{The sets in Example 3.2}
        \label{fig:finiteexample}
    \end{figure}

    Next, we aim to identify the geometry of $h(\mathbb{H})$. We do this by inspecting the boundary behavior of $h$. Differentiating term by term, we see that $h^{\prime}(z)=\sum_{n=1}^{+\infty}\frac{1}{n}\frac{1}{z-n}-1$ and $h''(z)=-\sum_{n=1}^{+\infty}\frac{1}{n}\frac{1}{(z-n)^2}$. So, for $x\in\mathbb{R}\setminus\{n : n\ge1\}$, we have $h^{\prime \prime} (x)<0$. This signifies that $h'$ is strictly decreasing in every interval of the form $(m,m+1)$, $m\in\mathbb{N}$, and in the interval $(-\infty,1)$. However $\lim_{x\to m^+}h'(x)=+\infty$ and $\lim_{x\to m^-}h'(x)=-\infty$, for every $m\in\mathbb{N}$. Therefore, for each $m\in\mathbb{N}$ there exists some unique $k_m\in(m,m+1)$ so that $h'(k_m)=0$. Combining, we see that for each $m\in\mathbb{N}$, $h'(x)>0$ for $x \in (m,k_m)$, $h'(x)<0$ for $x \in (k_m,m+1)$ and also $h'(x)<0$ for $x \in (-\infty,1)$. Considering that $h'(x)=\RE h'(x)$ and $h''(x)=\RE h''(x)$, but $h(x)$ might have non zero imaginary part, all the above information may be translated into information about the monotonicity of $\RE h(x)$. As a matter of fact, $\RE h(x)$ is decreasing in $(-\infty,1)$, increasing in every interval $(m,k_m)$ and decreasing in every interval $(k_m,m+1)$. To end the study of the real part of $h$ observe that $\RE h(m)=-\infty$ for all $m\in\mathbb{N}$ and $\lim_{x\to-\infty}h(x)=+\infty$.

    Moving on to the imaginary part, we see that for $x\in(m,m+1)$
    $$\IM h(x)=\sum\limits_{n=1}^{+\infty}\frac{1}{n}\Arg\left(\frac{n-x}{n-i}\right)=\sum\limits_{n=1}^{+\infty}\frac{1}{n}\Arg(n+i)-\sum\limits_{n=1}^{m}\frac{\pi}{n},$$
    with $y:=\sum_{n=1}^{+\infty}\frac{1}{n}\Arg(n+i)=\sum_{n=1}^{+\infty}\frac{1}{n}\arctan\frac{1}{n}<+\infty$. So in each interval $(m,m+1)$, $m\in\mathbb{N}$, the mapping $h$ maintains constant imaginary part. Clearly, for $x\in(-\infty,1)$ we see that $\IM h(x)=y$. Combining everything together, we understand that
    $$\Omega:=h(\mathbb{H})=\left\{w\in\mathbb{C}:\IM w<y\right\}\setminus\bigcup\limits_{m=1}^{+\infty}\left\{w\in\mathbb{C}:\IM w=y-\sum\limits_{n=1}^{m}\frac{\pi}{n}, \quad\RE w\le\RE h(k_m)\right\}.$$
    Certainly $\Omega$ is a convex in the positive direction domain which is contained inside a horizontal half-plane, but not a horizontal strip. Therefore, it is the Koenigs domain of a parabolic semigroup $(\phi_t)$ of positive hyperbolic step. In addition, due to the fact that the harmonic series diverges, $\Omega$ does not contain any horizontal half-plane. All that remains is to show  that the associated $(\phi_t)$ is of finite shift.

    Consider $\tau$ to be the Denjoy--Wolff point of $(\phi_t)$ and fix $z\in\D$. Through a Cayley transform $C:\D\to\mathbb{H}$ that maps $\tau$ to $\infty$, we map $z$ to $C(z)$. It is clear that $h\circ C$ can be selected to be the Koenigs function of $(\phi_t)$. So, every horizontal half-line contained inside $\Omega$ and stretching to the right is the image of some orbit of $(\phi_t)$. Recall that the horodisks of $\D$ tangent at $\tau$ are mapped through $C$ to horizontal half-planes. For $\delta>y-\IM h(C(z))>0$ consider the line $L_{\delta}:=\{x+i\delta: x\in\mathbb{R}\}$ which is the boundary of some horodisk $E_\delta$ of $\mathbb{H}$. Then
    $$\IM h(x+i\delta)=\sum\limits_{n=1}^{+\infty}\frac{1}{n}\Arg\left(\frac{n-x-i\delta}{n-i}\right)-\delta\le\sum\limits_{n=1}^{+\infty}\frac{1}{n}\arctan\left(\frac{1}{n}\right)-\delta=y-\delta<\IM h(C(z)),$$
    for all $x\in\mathbb{R}$. As a result, the half-line $\{h(C(z))+t:t\ge0\}$ does not intersect $h(E_\delta)$. This implies that the image through $h^{-1}$ of this half-line does not intersect $E_\delta$, which in turn implies that $\gamma_z$ does not intersect the horodisk $E(\tau,\frac{1}{\delta})$. In conclusion, $(\phi_t)$ is of finite shift and serves as the desired counterexample.
\end{example}

Next, we turn to parabolic semigroups of positive hyperbolic step and infinite shift. Even though in the case of finite shift the situation is prescribed, for semigroups of infinite shift the value of the inner argument may vary. Indeed, for each $\Theta \in [0,\pi]$, we may construct a semigroup of this type with inner argument $\Theta$. We will verify this through a series of examples. 

\begin{example}
    Let $\Omega=\{w=x+iy\in\mathbb{C}:x>0,\; 0<y<\sqrt{x}\}-1-i$. Then $\Omega$ is a convex in the positive direction simply connected domain with $0\in\Omega$. Therefore, $\Omega$ is the Koenigs domain of some non-elliptic semigroup $(\phi_t)$. Obviously $\Omega$ is contained inside a horizontal half-plane, but not inside a horizontal strip, so $(\phi_t)$ must be parabolic of positive hyperbolic step. It is easy to check that $\Omega$ contains no upper horizontal angular sector, so $\Theta=0$ and $(\phi_t)$ is necessarily of infinite shift. 
\end{example}

\begin{example}
    Let $\theta\in(0,\pi)$ and consider the upper horizontal angular sector $\Omega=\{w\in\mathbb{C}:\Arg w\in(0,\theta)\}-e^{i\frac{\pi}{4}}$. Following a similar reasoning as above, $\Omega$ corresponds to a parabolic semigroup $(\phi_t)$ of positive hyperbolic step. Evidently $\Theta=\theta\in(0,\pi)$ and $(\phi_t)$ is necessarily of infinite shift.
\end{example}

\begin{example}\cite[Example 1.2]{cordella}
    Let $\Omega=\{w=x+iy\in\mathbb{C}:y>0,\;x>-y|\log y|\}-2i$. In \cite{cordella}, the author shows that this domain corresponds to a parabolic semigroup $(\phi_t)$ of positive hyperbolic step and infinite shift. It can be readily checked that $\Omega$ contains a translation of an upper horizontal angular sector $S_\theta$, for every $\theta\in(0,\pi)$. As a consequence, $\Theta=\pi$.
\end{example}

Before concluding the section, one last useful observation stemming from the definition is the translation of finite shift in the setting of a half-plane. We use the upper half-plane $\mathbb{H}$ as a benchmark. Let $\tau\in\partial\D$ and consider the horodisk $E(\tau,R)$, for some $R>0$. Then, the Cayley transform $C:\D\to \mathbb{H}$, with $C(z)=i\frac{\tau+z}{\tau-z}$ mapping $\tau$ to $\infty$ sends every horodisk of $\D$ to a half-plane parallel to $\mathbb{H}$. To be more exact, it is easily computed that $C(E(\tau,R))=H_{\frac{1}{R}}=\{\zeta\in\mathbb{C}:\IM \zeta>\frac{1}{R}\}$. Therefore, combining this remark with our knowledge of finite shift semigroups, we see that a semigroup $(\phi_t)$ is of finite shift if and only if for every $z\in\D$ there exists some $R_z>0$ so that $\IM C(\phi_t(z))=\IM C(\gamma_z(t))<\frac{1}{R_z}$, for all $t\ge0$.

\section{Rates of Convergence}\label{sec:rates}
Having made all the preparatory work, we are ready to proceed to the main body of the article. In the current section, we prove our main theorems concerning the rates by which the orbits of a semigroup converge to the Denjoy--Wolff point in terms of Euclidean and hyperbolic geometries.


\begin{proof}[\bf Proof of Theorem \ref{theo-hyper}]
 (i) Let $h$ be the Koenigs function of $(\phi_t)$ and $\Omega$ the respective Koenigs domain. Fix $z\in\D$ and let $\epsilon>0$. Consider $\theta=\frac{\Theta\pi}{\pi+2\epsilon\Theta}$. It can be readily checked that $\theta\in(0,\Theta)$. By the conformal invariance of the hyperbolic distance, we have $d_{\D}(z,\phi_t(z))=d_{\Omega}(h(z),h(\phi_t(z)))=d_{\Omega}(h(z),h(z)+t)$. First of all, since $\theta$ is smaller than the inner argument of $(\phi_t)$ we may find a point $p_\theta\in\Omega$ and $t_0>0$ ($p_\theta$ and $t_0$ depend on $z$ and $\epsilon$) such that $h(z)+t$ is contained in $p_\theta+S_{\theta}\subset\Omega$, for $t\ge t_0$.
 
Moreover, the semigroup must be also parabolic of positive hyperbolic step, so there exists $\rho<0$ so that
    $\Omega\subset H_{\rho}$. Combining everything and through the triangle inequality and the domain monotonicity property of the hyperbolic distance, we get
    \begin{equation}\label{monotonicity}
        d_{H_\rho}(h(z),h(z)+t)\le d_{\Omega}(h(z),h(z)+t)\le d_\Omega(h(z),h(z)+t_0)+d_{p_\theta+S_{\theta}}(h(z)+t_0,h(z)+t),
    \end{equation}
    for all $t\ge t_0$. For a fixed $\epsilon>0$, the hyperbolic distance $d_\Omega(h(z),h(z)+t_0)$ is bounded and can just be treated as an additive constant depending on both $z$ and $\epsilon$. So, we may assume beforehand that $t_0=0$. We treat each side of (\ref{monotonicity}) separately.

    We start with the easier left-hand side. We need to evaluate the hyperbolic distance in the upper half-plane $H_{\rho}$. Then, by (\ref{half-plane})
    \begin{eqnarray*}
        d_{H_\rho}(h(z),h(z)+t)&=&\frac{1}{2}\log\frac{|2i\IM h(z)-t-2\rho i|+t}{|2i\IM h(z)-t-2\rho i|-t}\\
        &=&\frac{1}{2}\log\frac{t^2+4(\IM h(z)-\rho)^2+t^2+2t\sqrt{t^2+4(\IM h(z)-\rho)^2}}{4(\IM h(z)-\rho)^2}\\
        &\ge&\frac{1}{2}\log\frac{4t^2}{4(\IM h(z)-\rho)^2}\\
        &=&\log t-\log(\IM h(z)-\rho).
    \end{eqnarray*}
    So, going back to the unit disk, we find $d_{\D}(z,\phi_t(z))\ge\log t-\log(\IM h(z)-\rho)$ and this inequality is in fact true for all $t\ge0$.

    Next, we must work with the right-hand side of (\ref{monotonicity}), which is carried out through estimations with the hyperbolic distance inside the horizontal angular sector $p_\theta+S_{\theta}$. Set $w=h(z)-p_\theta$, $w_1(t)=|w+t|^2$, $w_2(t)= \Arg(w+t)$, $d_1=\RE(w^{\pi / \theta})$, $d_2=\IM(w^{\pi / \theta})$, 
      $$A_1(t): =
      \left[\left(w_1(t)^{\pi/ 2\theta} \cos\left(\frac{\pi}{\theta} w_2(t) \right) -d_1 \right)^2 + \left(w_1(t)^{\pi/ 2\theta} \sin\left(\frac{\pi}{\theta} w_2(t) \right) -d_2 \right)^2 \right]^{1/2}$$ and $$A_2(t):= 
      \left[\left(w_1(t)^{\pi/ 2\theta} \cos\left(\frac{\pi}{\theta} w_2(t) \right) -d_1 \right)^2 + \left(w_1(t)^{\pi/ 2\theta} \sin\left(\frac{\pi}{\theta} w_2(t) \right) +d_2 \right)^2 \right]^{1/2}. $$
        Due to conformal invariance, the domain monotonicity property of the hyperbolic distance and \eqref{sector}, it follows that 
        \begin{eqnarray*}
            d_{\D}(z, \phi_t(z)) &\leq&d_{p_\theta+S_\theta}(h(z),h(z)+t) \\
            &=&d_{S_\theta}(w,w+t)\\
            &=&\frac{1}{2}   \log \frac{A_2(t)+A_1(t)}{A_2(t) -A_1(t)} \\
            &=&\frac{1}{2}      \log \frac{(A_1(t)+A_2(t))^2}{A_2(t)^2 - A_1(t)^2} \\
           &= &\frac{1}{2}   \log\left[ \frac{w_1(t)^{\pi/\theta}}{4d_2 w_1(t)^{\pi/ 2\theta}  \sin \left(\frac{\pi}{\theta} w_2(t) \right) } \left( \frac{A_1(t)}{w_1(t)^{\pi/2\theta}} + \frac{A_2(t)}{w_1(t)^{\pi/2\theta}} \right)^2\right] ,
           \end{eqnarray*}
       for all $t\ge 1$. Next, we may find some $t_1\ge 1$ so that 
       \begin{equation}\label{eq:w1}
       1\le w_1(t) \le 2t^2,
       \end{equation}
       for all $t\ge t_1$. Through this inequality we can see that      
        \begin{eqnarray*}
            A_i(t) w_1(t)^{-\pi/2\theta} &\leq& \left[\left(\left|\cos\left(\frac{\pi}{\theta} w_2(t) \right)\right| + \frac{|d_1|}{w_1(t)^{\pi/2\theta}} \right)^2 + \left( \left|\sin\left(\frac{\pi}{\theta} w_2(t) \right) \right| +\frac{|d_2|}{w_1(t)^{\pi/ 2\theta} } \right)^2 \right]^{1/2} \\
            &\le & 2+\frac{1}{w_1(t)^{\pi / 2\theta}}(|d_1|+|d_2|)\\
            &\underset{\eqref{eq:w1}}{\leq}& 2 +2|w|^{\pi / \theta}=:\alpha
        \end{eqnarray*}
for $i=1,2$, for all $t\ge t_1$, where $\alpha =\alpha(z,\epsilon)$.
        As a result, combining everything we get
    \begin{eqnarray}
        \notag d_{\D}(z,\phi_t(z))&\le&\frac{1}{2}\log w_1(t)^{\pi / 2\theta}-\frac{1}{2}\log\sin\left(\frac{\pi}{\theta}w_2(t)\right)+\frac{1}{2}\log \alpha^2 \\
        \label{eq:upperest} &\underset{\eqref{eq:w1}}{\leq}& \frac{1}{2}\log t^{\pi / \theta}-\frac{1}{2}\log\sin\left(\frac{\pi}{\theta}w_2(t)\right)+\frac{1}{2}\log\left(a^22^{\pi/ 2\theta}\right),
    \end{eqnarray}
    for all $t\ge t_1$. By the shape of the domains, we have that $\IM w>0$ and therefore $\lim_{t\to+\infty}w_2(t)=0$ with the convergence happening through positive values. Therefore, there exists some $t_2\ge t_1$ such that $\sin(\frac{\pi}{\theta}w_2(t))\ge\frac{\pi}{2\theta}w_2(t)=\frac{\pi}{2\theta}\arctan\frac{\IM w}{\RE w+t}$, for all $t\ge t_2$. In a similar fashion, we see $\lim_{t\to+\infty}\arctan\frac{\IM w}{\RE w+t}=0$ and hence we may find $t_3\ge \max\{t_2,|\RE w|\}$ so that $\arctan\frac{\IM w}{\RE w+t}\ge\frac{\IM w}{2(\RE w+t)}$, for all $t\ge t_3$. As a consequence,
    \begin{eqnarray*}
        -\frac{1}{2}\log\sin\left(\frac{\pi}{\theta}w_2(t)\right)&\le&-\frac{1}{2}\log\left(\frac{\pi}{4\theta}\frac{\IM w}{\RE w+t}\right)\\
        &=&\frac{1}{2}\log(\RE w+t)-\frac{1}{2}\log\left(\frac{\pi\IM w}{4\theta}\right)\\
        &\le&\frac{1}{2}\log t-\frac{1}{2}\log\left(\frac{\pi\IM w}{2\theta}\right),
    \end{eqnarray*}
    for all $t\ge t_3$.
    Returning back to \eqref{eq:upperest}, we obtain that there exists a constant $c_0=c_0(z,\theta)$ such that 
    $$ d_{\D}(z, \phi_t(z)) \leq \frac{\pi +\theta }{2\theta} \log t+c_0=\left(\frac{\pi+\Theta}{2\Theta}+\epsilon\right)\log t+c_0 \, , \quad t\geq t_3.$$  
    Finally, the hyperbolic length $l_{\D}(\gamma_z;[1,t_3])$ is obviously finite, so we may trivially write
    $$d_{\D}(z,\phi_t(z))\le\left(\frac{\pi+\Theta}{2\Theta}+\epsilon\right)\log t+ l_{\D}(\gamma_z;[1,t_3]),$$
    for all $t\in(1,t_3)$. Setting $c_2=c_2(z,\epsilon)=\max\{c_0,l_{\D}(\gamma_z;[1,t_3])\}$, we deduce the desired result.

(ii) Let us now handle the case of finite shift. The lower bound is deduced exactly as in the previous case. We only need to deal with the upper bound. Fix $z\in\D$. Through the Cayley transform $C:\D\to\mathbb{H}$ with $C(z)=i\frac{\tau+z}{\tau-z}$, we map the orbit $\gamma_z$ to an orbit $\gamma:[0,+\infty)\to\mathbb{H}$ in the upper half-plane, where $\gamma(t):=C(\gamma_z(t))=C(\phi_t(z))$. By the conformal invariance of the hyperbolic distance
$$d_{\D}(z,\phi_t(z))=d_\mathbb{H}(C(z),C(\phi_t(z)))=d_\mathbb{H}(\gamma(0),\gamma(t)).$$
For the sake of brevity, set $\gamma(t)=x_t+iy_t$. By Julia's Lemma, we know that $y_t$ is strictly increasing for $t\ge0$. Moreover, since $(\phi_t)$ is of finite shift, we already established that $\lim_{t\to+\infty}y_t=L\in(0,+\infty)$. So $y_t<L$, for all $t\ge0$. Furthermore, the orbit $\gamma_z$ converges to $\tau$ tangentially. Translating this piece of information in the setting of the upper half-plane, it can be easily checked that $\lim_{t\to+\infty}|x_t|=+\infty$. But the convergence to $\tau$ can be either only by angle $0$, or only by angle $\pi$. Again, passing to $\mathbb{H}$, this means that either $\lim_{t\to+\infty}x_t=-\infty$ or $\lim_{t\to+\infty}x_t=+\infty$. Without loss of generality, we may assume the latter holds. Therefore, there exists some $t_0\ge0$ such that $x_t>x_0$, for all $t\ge t_0$. Utilizing the triangle inequality, we may write
$$d_{\D}(z,\phi_t(z))=d_\mathbb{H}(x_0+iy_0,x_t+iy_t)\le d_\mathbb{H}(x_0+iy_0,x_t+iy_0)+d_\mathbb{H}(x_t+iy_0,x_t+iy_t).$$
Through formula (\ref{half-plane}), it is straightforward that $d_\mathbb{H}(x_t+iy_0,x_t+iy_t)\le\frac{1}{2}\log\frac{L}{y_0}=:c_0$, for all $t\ge0$. So, as of yet,
\begin{equation}\label{first bound}
    d_{\D}(z,\phi_t(z))\le d_\mathbb{H}(x_0+iy_0,x_t+iy_0)+c_0, \quad\quad t\ge0.
\end{equation}
Now we have to evaluate the hyperbolic distance in $\mathbb{H}$ of a horizontal rectilinear segment which is seemingly easier to work with. Restricting ourselves to $t\ge t_0$, we have $x_t-x_0>0$. Again using formula (\ref{half-plane}) and executing similar computation as in the previous case of the theorem, we may prove that
$$d_\mathbb{H}(x_0+iy_0,x_t+iy_0)=\frac{1}{2}\log\frac{2y_0^2+(x_t-x_0)^2+(x_t-x_0)\sqrt{4y_0^2+(x_t-x_0)^2}}{2y_0^2},$$
for all $t\ge t_0$. Since $x_t\to+\infty$, there exists some $t_1\ge t_0$ such that $2y_0^2\le (x_t-x_0)^2$, for all $t\ge t_1$. As a result,
\begin{eqnarray}\label{second bound}
\notag    d_\mathbb{H}(x_0+iy_0,x_t+iy_0)&\le&\frac{1}{2}\log\frac{2(x_t-x_0)^2+(x_t-x_0)\sqrt{3(x_t-x_0)^2}}{2y_0^2}\\
 \notag   &=&\frac{1}{2}\log\frac{(2+\sqrt{3})(x_t-x_0)^2}{2y_0^2}\\
\notag    &=&\log(x_t-x_0)+\frac{1}{2}\log\frac{2+\sqrt{3}}{2y_0^2}\\
    &\le&\log|\gamma(t)-\gamma(0)|+\frac{1}{2}\log\frac{2+\sqrt{3}}{2y_0^2},
\end{eqnarray}
for all $t\ge t_1$. Returning to (\ref{first bound}), we get $d_{\D}(z,\phi_t(z))\le\log|\gamma(t)-\gamma(0)|+C_0$, where $C_0=c_0+\frac{1}{2}\log\frac{2+\sqrt{3}}{2y_0^2}$. We are left with estimating the modulus $|\gamma(t)-\gamma(0)|$. We will achieve this by proving that $\gamma$ is actually Lipschitz. Through quick calculation, we see that $\frac{\partial\gamma(t)}{\partial t}=\frac{2\tau\frac{\partial\phi_t(z)}{\partial t}}{(\tau-\phi_t(z))^2}=\frac{2\tau G(\phi_t(z))}{(\tau-\phi_t(z))^2}$, where $G$ is the associated infinitesimal generator. Denoting by $h$ the Koenigs function of $(\phi_t)$, we know that $G(\phi_t(z))=\frac{1}{h'(\phi_t(z))}$ and therefore 
$$\left|\frac{\partial\gamma(t)}{\partial t}\right|=\frac{2}{|h'(\phi_t(z))|\cdot|\tau-\phi_t(z)|^2}.$$
However, by a well-known result concerning conformal mapping (see e.g. \cite[Corollary 1.4]{pommerenke4}), we have $|h'(\phi_t(z))|\ge\frac{\text{dist}(h(\phi_t(z)),\partial\Omega)}{1-|\phi_t(z)|^2}$, where $\Omega=h(\D)$ is the Koenigs domain of $(\phi_t)$. Combining, we get
$$\left|\frac{\partial\gamma(t)}{\partial t}\right|\le\frac{2(1-|\phi_t(z)|^2)}{\text{dist}(h(z)+t,\partial\Omega)\cdot|\tau-\phi_t(z)|^2}.$$
However, $(\phi_t)$ is of finite shift. Consequently, by definition, there exists some $R_z>0$ such that $\frac{1-|\phi_t(z)|^2}{|\tau-\phi_t(z)|^2}\le R_z$, for all $t\ge0$. In addition, $\Omega$ is convex in the positive direction. Thus, trivially, $\text{dist}(h(z)+t,\partial\Omega)\ge\text{dist}(h(z),\partial\Omega)=:d$, for all $t\ge0$. So, for all $t\ge0$
$$\left|\frac{\partial\gamma(t)}{\partial t}\right|\le\frac{2R_z}{d}=:M.$$
Therefore, $\gamma$ is Lipschitz and $|\gamma(t)-\gamma(0)|\le M(t-0)$, for all $t\ge0$. Going back, we are led to $d_{\D}(z,\phi_t(z))\le\log t+\log M+C_0$, for all $t\ge t_1$. Arguing as in the previous case, there exists some constant $c_2=c_2(z)$ so that the result holds for all $t>1$.

\end{proof}

\begin{corollary}\label{cor:limitshyp}
    Let $(\phi_t)$ be a semigroup in $\D$ of positive hyperbolic step with inner argument $\Theta \in (0,\pi]$. The following are true:
    \begin{itemize}
    \item[(i)]If $(\phi_t)$ is of infinite shift, then for all $z\in\D$
    $$1\le\liminf\limits_{t\to+\infty}\frac{d_{\D}(z,\phi_t(z))}{\log t}\le\limsup\limits_{t\to+\infty}\frac{d_{\D}(z,\phi_t(z))}{\log t}\le\frac{\pi+\Theta}{2\Theta}.$$
    \item[(ii)] If $(\phi_t)$ is of finite shift (and hence $\Theta=\pi$), then for all $z\in\D$
    $$\lim\limits_{t\to+\infty}\frac{d_{\D}(z,\phi_t(z))}{\log t}=1.$$
    \end{itemize}
    
   \begin{proof}
    We start with (i). The left-hand side inequality about the limit infimum can be deduced directly from Theorem \ref{theo-hyper}. For the right-hand side, Theorem \ref{theo-hyper} yields $\frac{d_{\D}(z,\phi_t(z))}{\log t}\le\frac{\pi+\Theta}{2\Theta}+\epsilon+\frac{c_1(z,\epsilon)}{\log t}$, for all $t>1$ and all $\epsilon>0$. This leads to $\limsup_{t\to+\infty}\frac{d_{\D}(z,\phi_t(z))}{\log t}\le\frac{\pi+\Theta}{2\Theta}+\epsilon$, for all $\epsilon>0$ and therefore, we immediately get the desired result. A similar procedure shows that in case (ii), the corresponding limit actually exists and is equal to $1$.
\end{proof}
\end{corollary}

Having completed our study on the rate of ``divergence'' of the hyperbolic distance, we may now proceed to the study of the quantity $|\phi_t(z)-\tau|$, where $\tau\in\partial\D$ is the Denjoy--Wolff point of the semigroup $(\phi_t)$. We will first provide a lemma that correlates the quantity $|\phi_t(z)-\tau|$ of any semigroup $(\phi_t)$ of infinite shift (even if it is hyperbolic or parabolic of zero hyperbolic step) with the hyperbolic distance $d_{\D}(z,\phi_t(z))$. Then, this lemma combined with Theorem \ref{theo-hyper}(i) will lead to the proof of Theorem \ref{theo-eucl}(i). On the other hand, Theorem \ref{theo-eucl}(ii) will be derived through Theorem \ref{theo-hyper}(ii) and some further arguments.
\begin{lemma}\label{Lemma}
    Let $(\phi_t)$ be a semigroup of infinite shift in $\D$ with Denjoy-Wolff point $\tau$. Then, for all $z\in\D$ and all $t\ge0$ we have
    $$\frac{1-|z|}{1+|z|}e^{-2d_{\D}(z,\phi_t(z))}\le|\phi_t(z)-\tau|\le 2\frac{|\tau-z|}{1-|z|}e^{-d_{\D}(z,\phi_t(z))}.$$
    \begin{proof}
        We will need several inequalities to reach the desired conclusion. First of all, by the triangle inequality for the hyperbolic distance, we get
        \begin{equation}\label{Lemma4.1.1}
            d_{\D}(z,\phi_t(z))-d_{\D}(0,z)\le d_{\D}(0,\phi_t(z))\le d_{\D}(z,\phi_t(z))+d_{\D}(0,z),
        \end{equation}
    for all $t\ge0$ and all $z\in\D$. In addition, by the very formula of the hyperbolic distance in the unit disk, we quickly see that $(1-|\phi_t(z)|)=e^{-2d_{\D}(0,\phi_t(z))}(1+|\phi_t(z)|)$, which leads to
    \begin{equation}\label{Lemma4.1.2}
        e^{-2d_{\D}(0,\phi_t(z))}\le 1-|\phi_t(z)|\le 2e^{-2d_{\D}(0,\phi_t(z))},
    \end{equation}
    for all $t\ge0$ and all $z\in\D$. On top of that, by the triangle inequality, we have at once $|\phi_t(z)-\tau|\ge 1-|\phi_t(z)|$. On the other hand, since $(\phi_t)$ is of infinite shift, there exists a minimum $R_z$ such that $\phi_t(z)\in\overline{E(\tau,R_z)}$ for all $t\ge0$ ($\phi_0(z)=z$ will be on the boundary of this horodisk, while by Julia's Lemma, every other $\phi_t(z)$ will not leave the horodisk). By the definition of horodisks, this means that $|\phi_t(z)-\tau|^2\le R_z(1-|\phi_t(z)|^2)$. Combining, we get
    \begin{equation}\label{Lemma4.1.3}
        1-|\phi_t(z)|\le|\phi_t(z)-\tau|\le \sqrt{2R_z}\sqrt{1-|\phi_t(z)|},
    \end{equation}
    for all $t\ge0$ and all $z\in\D$. In particular, due to the fact that $z$ lies on the boundary of $E(\tau,R_z)$, it is true that $R_z=\frac{|\tau-z|^2}{1-|z|^2}$. Finally, through a successive application of (\ref{Lemma4.1.3}), (\ref{Lemma4.1.2}) and (\ref{Lemma4.1.1}), we deduce
    $$e^{-2d_{\D}(0,z)}e^{-2d_{\D}(z,\phi_t(z))}\le|\phi_t(z)-\tau|\le2\sqrt{R_z}e^{d_{\D}(0,z)}e^{-d_{\D}(z,\phi_t(z))}.$$
    Keeping in mind that $R_z=\frac{|\tau-z|^2}{1-|z|^2}$ and that $d_{\D}(0,z)=\frac{1}{2}\log\frac{1+|z|}{1-|z|}$, we reach the desired result.
    \end{proof}
\end{lemma}

\begin{proof}[\bf Proof of Theorem \ref{theo-eucl}]
    (i) A combination of Theorem \ref{theo-hyper}(i) and Lemma \ref{Lemma} leads directly to the desired result.
    
    (ii) By the hypothesis, $(\phi_t)$ is parabolic of positive hyperbolic step. So, the upper bound is already known by \cite{betsparsem} and \cite{Bets-Contreras}. Thus we only need to work towards the lower bound. Since $(\phi_t)$ is also of finite shift, for each $z\in\D$ there exists a maximum $R_z$ such that $\phi_t(z)\notin E(\tau,R_z)$, for all $t\ge0$. This signifies that $|\phi_t(z)-\tau|^2\ge R_z(1-|\phi_t(z)|^2)$ and hence $|\phi_t(z)-\tau|\ge\sqrt{2R_z}\sqrt{1-|\phi_t(z)|}$, for all $t\ge0$ and all $z\in\D$. By the formula for the hyperbolic distance in the unit disk, we observe
\begin{eqnarray*}
    1-|\phi_t(z)|&=&e^{-2d_{\D}(0,\phi_t(z))}(1+|\phi_t(z)|)\\
    &\ge& e^{-2d_{\D}(0,\phi_t(z))}\\
    &\ge&e^{-2d_{\D}(0,z)}e^{-2d_{\D}(z,\phi_t(z))}\\
    &=&\frac{1-|z|}{1+|z|}e^{-2d_{\D}(z,\phi_t(z))},
\end{eqnarray*}
where in the second to last relation we made use of the triangle inequality. Therefore, we are led to
$$|\phi_t(z)-\tau|\ge\sqrt{2R_z}\sqrt{\frac{1-|z|}{1+|z|}}e^{-d_{\D}(z,\phi_t(z))}.$$
Theorem \ref{theo-hyper}(ii) dictates that for each $z\in\D$ there exists some constant $c>0$ depending on $z$ such that
$$d_{\D}(z,\phi_t(z))\le\log t+c,$$
for all $t>1$. Returning to the previous inequality, we understand that
$$|\phi_t(z)-\tau|\ge\sqrt{2R_z}\sqrt{\frac{1-|z|}{1+|z|}}e^{-c}e^{-\log t},$$
for all $t>1$ and all $z\in\D$. Setting $c_1=\sqrt{2R_z}\sqrt{\frac{1-|z|}{1+|z|}}e^{-c}$ which depends on $z$, we obtain the desired rate of convergence.
\end{proof}

\section{ Rates of Convergence - Harmonic Measure}\label{sec:harmmeas}

The current section is devoted to the rate of convergence in terms of the harmonic measure. First, we must carefully select the circular arc $E\subset\partial\D$ with regard to which we estimate the harmonic measure. Since each orbit $\gamma_z$ of the semigroup converges to the Denjoy--Wolff point $\tau$, the selection of $E$ has to occur in such a way that one of its endpoints is $\tau$. An obvious choice is to consider $E$ to be one of the half-circles defined by $\tau$ and $-\tau$. However, the harmonic measure of $\gamma_z(t)$ with respect to any set $B\subset\partial\D$ located ``far'' from $\tau$ tends to $0$, as $t\to+\infty$. As a result, any choice of a circular arc with one endpoint at $\tau$ may work for our purposes. For this reason, we may assume without loss of generality that given $z\in\D$, $E_1$ is the open circular arc corresponding through $h$ to the boundary set $\partial\Omega^-$ containing all the prime ends of $\Omega$ defined by crosscuts with imaginary parts less than $\IM h(z)$. Clearly, we take $E_2$ to be the open circular arc corresponding through $h$ to the boundary set $\partial\Omega^+$ containing all the prime ends of $\Omega$ defined by crosscuts with imaginary parts greater than $\IM h(z)$. For further information on prime ends see \cite[Chapter 2]{pommerenke4}.

\begin{proof}[\bf Proof of Theorem \ref{theo-harm}.]
    Let $h$ be the Koenigs function of $(\phi_t)$ and set $\Omega:=h(\D)$. Fix $z\in\D$. The trajectory of $z$ is mapped through $h$ onto the half-line $\{h(z)+t:t\ge0\}$. Since $(\phi_t)$ has inner argument $\Theta$, for each fixed $\theta\in(0,\Theta)$ there exists $w_\theta\in\Omega$ such that the angular sector $V:=w_{z,\theta}+S_\theta$ is contained inside $\Omega$, while also containing $h(z)+t$, for sufficiently large $t\ge0$. Without loss of generality, we may assume that $h(z)+t\in V$, for all $t\ge0$. Moreover, since $(\phi_t)$ is also parabolic of positive hyperbolic step, there exists a horizontal half-plane $H$ such that $\Omega\subset H$. 

    Through $h$, the circular arcs $E_1$ and $E_2$ correspond (in the sense of prime ends) to the two boundary components $\partial\Omega^-$ and $\partial\Omega^+$, respectively. Then, by the conformal invariance of harmonic measure,
    $$\omega(\phi_t(z),E_2,\D)=\omega(h(z)+t,\partial\Omega^+,\Omega),$$
    for all $t\ge0$. Then, since harmonic measure is a Borel probability measure,
    $$\omega(h(z)+t,\partial\Omega^+,\Omega)=1-\omega(h(z)+t,\partial\Omega^-,\Omega),$$
    because $\partial\Omega^+$ and $\partial\Omega^-$ are disjoint in the sense of prime ends. We denote by $\hat{\partial\Omega^-}$ the set of boundary points in $\partial\Omega$ which correspond to the prime ends in $\partial\Omega^-$. A similar consideration is made for the set $\hat{\partial\Omega^+}$.

    \textbf{Case 1:} $\overline{\hat{\partial\Omega^+}}\cap\overline{\hat{\partial\Omega^-}}=\emptyset$

    Consider $\Omega^*$ to be the simply connected domain with boundary $\partial\hat{\Omega}^+\cup\partial H$. By the maximum principle for harmonic functions, we have
    $$\omega(h(z)+t,\partial\Omega^-,\Omega)\ge\omega(h(z)+t,\partial H,\Omega^*).$$
    Now denote by $\partial V^-$ the horizontal side of the angular sector $V$ and by $\partial V^+$ the other side. Surely, the extension of $\partial V^+$ intersects $\partial H$ at some point $\zeta_0$, thus separating $\partial H$ into two new components. Denote by $\partial H^+$ the one to the right and by $\partial H^-$ the remaining one. Since $\partial H^+\subset\partial H$, it is obvious that
    $$\omega(h(z)+t,\partial H,\Omega^*)\ge\omega(h(z)+t,\partial H^+,\Omega^*).$$

 \begin{figure}
        \centering
        \includegraphics[scale=0.6]{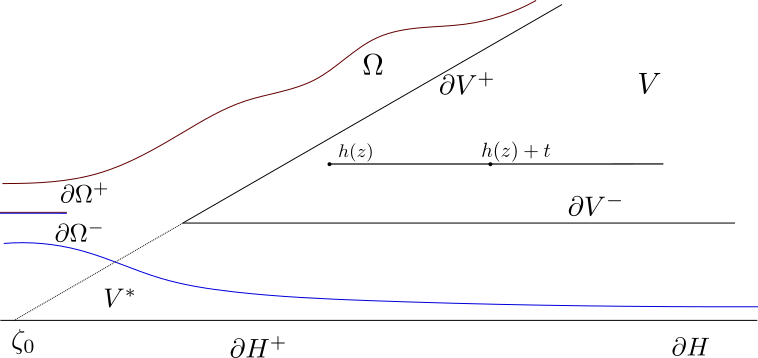}
        \caption{The construction in the proof of Theorem \ref{theo-harm}}
        \label{fig:harmonic}
    \end{figure}

    Lastly, set $V^*$ the angular sector formed by $\partial V^+$, its extension and $\partial H^+$. Clearly, this sector is of angle $\theta$, as well. In fact, $V^*\subset\Omega^*$, while $\partial H^+\subset\partial V^*\cap\partial\Omega^*$. Therefore, by the domain monotonicity property of the harmonic measure, we get
    $$\omega(h(z)+t,\partial H^+,\Omega^*)\ge\omega(h(z)+t,\partial H^+,V^*).$$
    Combining everything, we get 
    $$\omega(h(z)+t,\partial\Omega^-,\Omega)\ge\omega(h(z)+t,\partial H^+,V^*),$$
    which in turn leads to
    $$\omega(h(z)+t,\partial\Omega^+,\Omega)\le1-\omega(h(z)+t,\partial H^+,V^*),$$
    for all $t\ge0$. Then, again through conformal invariance, we have
    \begin{eqnarray*}
        \omega(h(z)+t,\partial H^+,V^*)&=&\omega(h(z)+t-\zeta_0,\partial H^+-\zeta_0,V^*-\zeta_0)\\
        &=&\omega(h(z)+t-\zeta_0,[0,+\infty),\{w\in\mathbb{C}:0<\arg w<\theta\})\\
        &=&\frac{\theta-\arg(h(z)+t-\zeta_0)}{\theta}\\
        &=&\frac{\theta-\arctan\frac{y}{x+t}}{\theta},
    \end{eqnarray*}
    where $h(z)-\zeta_0=x+iy$ and $y>0$. As a result,
    $$\omega(h(z)+t,\partial\Omega^+,\Omega)\le1-\frac{\theta-\arctan\frac{y}{x+t}}{\theta}=\frac{\arctan\frac{y}{x+t}}{\theta},$$
    for all $t\ge0$. Multiplying by $t$ and taking limits, we are led to
    $$\limsup\limits_{t\to+\infty}\left(t\cdot\omega(\phi_t(z),E_2,\D)\right)\le\lim\limits_{t\to+\infty}\frac{t\cdot\arctan\frac{y}{x+t}}{\theta}=\frac{y}{\theta}\in(0,+\infty),$$
    for every $\theta\in(0,\Theta)$. However, we understand that $y=\text{dist}(h(z),\partial H)$ which remains the same, regardless of $\theta$. So, we can get 
    $$\limsup\limits_{t\to+\infty}\left(t\cdot\omega(\phi_t(z),E_2,\D)\right)\le\frac{\text{dist}(h(z),\partial H)}{\Theta},$$
    which proves the desired result.

    \textbf{Case 2:} $\overline{\hat{\partial\Omega^+}}\cap\overline{\hat{\partial\Omega^-}}=\{w_0\}$, a unique boundary point

    In this case, following a similar procedure, we consider $\Omega^*$ to be the simply connected domain bounded by $\partial\hat{\Omega}^+$, the half-line $\{w_0-t:t\ge0\}$ and $\partial H$. Then, the proof follows exactly the same steps, as before.

    \textbf{Case 3:} $\overline{\hat{\partial\Omega^+}}\cap\overline{\hat{\partial\Omega^-}}$ is more than a singleton (it could be a horizontal rectilinear segment or a horizontal half-line)

    In this case, we either enlarge or shrink one of the boundary components $\partial\hat{\Omega}^+,\partial\hat{\Omega}^-$ suitably, until their intersection in terms of boundary points is either empty or a singleton. Then, we continue as in the previous two cases.

    So, in any case, $\limsup_{t\to+\infty}\left(t\cdot\omega(\phi_t(z),E_2,\D)\right)\le\frac{\text{dist}(h(z),\partial H)}{\Theta}$.
\end{proof}

A natural first question is whether the rate in Theorem \ref{theo-harm} is the best possible in general. Again the answer is affirmative and we may see this through the following example.

\begin{example}
    Consider $\Omega=\{w:\IM w>-1\}\setminus\{w:\RE w\le-1,\IM =-\frac{1}{2}\}$. This is a simply connected and convex in the positive direction domain. Therefore, there exists a non-elliptic semigroup $(\phi_t)$ in $\D$ whose Koenigs domain is $\Omega$. Let $h:\D\to\Omega$ be the corresponding Koenigs function and recall that $h(0)=0$. In fact, $\Omega$ is contained inside a horizontal half-plane. This implies that the associated semigroup $(\phi_t)$ is parabolic of positive hyperbolic step. Furthermore, $\Omega$ also contains a horizontal half-plane and therefore $(\phi_t)$ must be of finite shift. Set 
    $$\partial\Omega^+=\left\{w:\RE w\le-1,\IM w=-\frac{1}{2}\right\} \quad\text{and}\quad\partial\Omega^-=\{w:\IM w=-1\}.$$
    As before, consider $E_1$ and $E_2$ to be the circular arcs corresponding through $h^{-1}$ to $\partial\Omega^-$ and $\partial\Omega^+$, respectively. A different selection would lead to the same result, but we pick the above configuration for the sake of convenience. Next, recall that $H_{-1}=\{w:\IM w>-1\}$ and $H_{-\frac{1}{2}}=\{w:\IM w>-\frac{1}{2}\}$. Obviously, $H_{-\frac{1}{2}}\subset\Omega\subset H_{-1}$. 

    We deal with the trajectory of $0$. By conformal invariance, $\omega(\phi_t(0),E_2,\D)=\omega(h(0)+t,\partial\Omega^+,\Omega)=\omega(t,\partial\Omega^+,\Omega)$. Certainly, $\partial\Omega^+\subset\partial\Omega\cap\partial H_{-\frac{1}{2}}$. As a result, by the domain monotonicity property of the harmonic measure, $\omega(t,\partial\Omega^+,\Omega)\ge\omega(t,\partial\Omega^+,H_{-\frac{1}{2}})$. Then, through known conformal mappings which leave the harmonic measure invariant, we find
    \begin{eqnarray*}
        \omega\left(t,\partial\Omega^+,H_{-\frac{1}{2}}\right)&=&\omega\left(t+\frac{i}{2},(-\infty,-1],\mathbb{H}\right)\\
        &=&\omega\left(t+1+\frac{i}{2},(-\infty,0],\mathbb{H}\right)\\
        &=&\frac{\arg(t+1+\frac{i}{2})-0}{\pi-0},
    \end{eqnarray*}
    for all $t\ge0$. Therefore, multiplying by $t$ and taking limits, we get
    $$\liminf\limits_{t\to+\infty}\left(t\cdot\omega(t,\partial\Omega^+,\Omega\right)\ge\lim\limits_{t\to+\infty}\frac{t\cdot\arctan\frac{1}{2(t+1)}}{\pi}=\frac{1}{2\pi},$$
    which provides the desired sharpness.
\end{example}

 Finally, one second question regarding this rate in terms of harmonic measure is whether it can be generalized for all parabolic semigroups of positive hyperbolic step, even if the inner argument is $0$. This time, the answer is negative. To give evidence of this fact, we construct one final example that demonstrates how a semigroup of infinite shift with $\Theta=0$ may fail to behave accordingly.

 \begin{example}
      Consider $$\Omega=\{w:\IM w>-1\}\setminus\bigcup_{n=1}^{+\infty}\{w:\RE w\le2^{2^n},\IM w=2^n\log2-1\}.$$ Clearly $\Omega$ is simply connected and convex in the positive direction. Hence, $\Omega$ can be regarded as the Koenigs domain of some non-elliptic semigroup $(\phi_t)$, contained also inside a horizontal half-plane. In addition, setting $a_n:=2^{2^n}+i(2^n\log2-1)$, we construct the sequence $\{a_n\}$ of the tip-points of the half-lines which form the upper boundary of $\Omega$. In fact, $\lim_{n\to+\infty}\IM a_n=+\infty$, and therefore $\Omega$ is not contained in any horizontal strip. So, $(\phi_t)$ is parabolic of positive hyperbolic step. On top of that, one can effortlessly check that $\Omega$ contains no horizontal angular sector. Hence $(\phi_t)$ is of infinite shift and has inner argument $\Theta=0$.

    Let $h$ denote the associated Koenigs function of $(\phi_t)$. Then, the half-line $[0,+\infty)$ is the image through $h$ of the trajectory of $0$, since $h(0)=0$. We set $\partial\Omega^-=\{w:\IM w=-1\}$, $\partial\Omega^+=\partial\Omega\setminus\partial\Omega^-$, $E_1=h^{-1}(\partial\Omega^-)$ and $E_2=h^{-1}(\partial\Omega^+)$. Then, as before, 
    $$\omega(\phi_t(0),E_2,\D)=\omega(h(0)+t,\partial\Omega^+,\Omega)=\omega(t,\partial\Omega^+,\Omega).$$
    Consider 
    $$t_n=\frac{2^{2^{n+1}}+2^{2^n}}{2}=\frac{2^{2^n}(2^{2^n}+1)}{2}=2^{2^n-1}(2^{2^n}+1).$$
    Then, $\{t_n\}\subset\mathbb[0,+\infty)$, $\lim_{n\to+\infty}t_n=+\infty$ and we deal with the quantity $\lim_{n\to+\infty}\omega(t_n,\partial\Omega^+,\Omega)$.

    Next, consider the rectangles $S_n=(2^{2^n},2^{2^{n+1}})\times(-1,2^{n+1}\log2-1)$ which satisfy $S_n\subset\Omega$, for all $n\in\mathbb{N}$. We denote by $U_n$ the upper side of $S_n$, by $D_n$ its lower side, by $R_n$ its right side and by $L_n$ its left side. Evidently $U_n\subset\partial\Omega^+$ and $D_n\subset\partial\Omega^-$, for all $n\in\mathbb{N}$. By construction, $t_n\in S_n$ for each $n\in\mathbb{N}$. Then, by known properties of harmonic measure, we see that
    $$\omega(t_n,\partial\Omega^+,\Omega)\ge\omega(t_n,U_n,\Omega)\ge\omega(t_n,U_n,S_n).$$
    We understand that the Euclidean distance of $t_n$ from the side $D_n$ is always equal to $-1$, while $\text{dist}(t_n,U_n)$ tends to $+\infty$, as $n\to+\infty$. Furthermore, the distance of $t_n$ from any of the vertical sides $R_n$, $L_n$, is equal to $\frac{2^{2^{n+1}}-2^{2^n}}{2}=2^{2^n-1}(2^{2^n}-1)$, which again diverges to $+\infty$. In addition, we may compute that
    $$\frac{\text{dist}(t_n,U_n)}{\text{dist}(t_n,R_n)}=\frac{\text{dist}(t_n,U_n)}{\text{dist}(t_n,L_n)}\to0\quad\text{and}\quad\frac{\text{length}(U_n)}{\text{length}(R_n)}=\frac{\text{length}(D_n)}{\text{length}(R_n)}\to+\infty,$$
    as $n\to+\infty$.
    As a consequence, through the shape and geometry of the rectangle $S_n$, we have $\omega(t_n,D_n,S_n)\to1$, while $\omega(t_n,R_n,S_n),\omega(t_n,L_n,S_n),\omega(t_n,U_n,S_n)\to0$, with the latter converging to $0$ with the slowest rate. So, for sufficiently large $n$, $\omega(t_n,R_n,S_n)\le\omega(t_n,U_n,S_n)$.

    Finally, let $\Sigma_n$ be the horizontal strip created through $S_n$ by extending the horizontal sides $U_n,D_n$. Set $\partial\Sigma_n^+$ and $\partial\Sigma_n^-$ the upper and lower, respectively, horizontal lines constituting the boundary of $\Sigma_n$. Then, $S_n\subset\Sigma_n$, whereas $U_n\subset\partial\Sigma_n^+$. Therefore,
    \begin{equation}\label{eq1}
    \omega(t_n,\partial\Sigma_n^+,\Sigma_n)=\omega(t_n,U_n,\Sigma_n)+\omega(t_n,\partial\Sigma_n^+\setminus U_n,\Sigma_n).
    \end{equation}
    By the maximum principle for harmonic functions, we see that
    \begin{equation}\label{eq2}
        \omega(t_n,\partial\Sigma_n^+\setminus U_n,\Sigma_n)\le\omega(t_n,R_n\cup L_n,S_n)=2\omega(t_n,R_n,S_n).
    \end{equation}
    In addition, by the Strong Markov Property for the harmonic measure
    \begin{eqnarray*}
        \omega(t_n,U_n,\Sigma_n)&=&\omega(t_n,U_n,S_n)+\int\limits_{\partial S_n\setminus\partial\Sigma_n}\omega(\zeta,U_n,\sigma_n)\cdot\omega(t_n,d_\zeta,S_n)\\
        &\le&\omega(t_n,U_n,S_n)+\omega(t_n,R_n\cup L_n,S_n)\\
        &=&\omega(t_n,U_n,S_n)+2\omega(t_n,R_n,S_n).
    \end{eqnarray*}

    Combining everything together, we have
    \begin{eqnarray*}
        \omega(t_n,\partial\Sigma_n^+,\Sigma_n)&\le&\omega(t_n,U_n,S_n)+2\omega(t_n,R_n,S_n)+2\omega(t_n,R_n,S_n)\\
        &\le&5\omega(t_n,U_n,S_n).
    \end{eqnarray*}
    But by known conformal mappings, we easily calculate $\omega(t_n,\partial\Sigma_n^+,\Sigma_n)=\frac{1}{2^{n+1}\log2}$ and therefore
    $$\omega(t_n,U_n,S_n)\ge\frac{1}{5\cdot2^{n+1}\log2}=\frac{1}{10\log2^{2^n}}\ge\frac{1}{10\log t_n},$$
    since $t_n\in(2^{2^n},2^{2^{n+1}})$. As a result, $\log t_n\cdot\omega(t_n,U_n,S_n)\ge \frac{1}{10}$ and going back, we understand that $\limsup_{t\to+\infty}(\log t\cdot\omega(\phi_t(0),E_2,\D))\ge \frac{1}{10}$, which in turn implies $\lim_{t\to+\infty}(t\cdot\omega(\phi_t(0),E_2,\D))=+\infty$.
 \end{example}

 \begin{remark}
    Studying the rate of convergence in terms of the Euclidean or the hyperbolic distance seems natural, whereas the ``harmonic rate'' might look abstract. However, the rate in terms of harmonic measure provides a great advantage in contrast to the other rates. Given a non-elliptic semigroup $(\phi_t)$ with Denjoy-Wolff point $\tau\in\partial\D$, we have $|\phi_t(z)-\tau|\to0$ and $d_{\D}(z,\phi_t(z))\to+\infty$, for all $z\in\D$. So the limit itself does not give any information, only the rate is relevant. On the other side, the limit of $\omega(\phi_t(z),E_2,\D)$ might not be the same for every $z\in\D$, while its actual value shows the angle by which $\gamma_z$ converges to $\tau$. Therefore, when looking at the inequality $\omega(\phi_t(z),E_2,\D)\le\frac{c}{t}$, without any assumption made on the semigroup, we not only find out the rate of the convergence, but we also understand that $\gamma_z$ converges tangentially to $\tau$.
\end{remark}

\section*{Declaration of competing interest}
None.

\section*{Acknowledgments}
We express our gratitude to Alan Sola for his suggestions and careful consideration of this work. Moreover, we thank Davide Cordella for his indications on the rate of convergence for semigroups of finite shift. 

\bibliographystyle{plain}

\end{document}